\RequirePackage{fix-cm}
\documentclass[smallextended]{svjour3}       
\smartqed  
\usepackage{amssymb,amsmath,amsthm}
\usepackage[top=1in, bottom=1in, inner=1in, outer=1in]{geometry}
\usepackage{setspace}
\newtheorem{notation}{Notation}
\newcommand*\xbar[1]{%
  \hbox{%
    \vbox{%
      \hrule height 0.5pt 
      \kern0.4ex
      \hbox{%
        \kern-0.15em
        \ensuremath{#1}%
        \kern-0.15em
      }%
    }%
  }%
}
\def\P{\mathbb{P} }

\def\R{\mathbb{R}}
\begin{document}
 \journalname{JOTP}

\newtheorem*{thma}{Theorem A}
\newtheorem*{thmb}{Theorem B}

\title{A generalization of the submartingale property:  maximal inequality and applications to various stochastic processes\thanks{This research was supported in part by Simons Foundation Grant  579110. The support is gratefully acknowledged.}
}

\titlerunning{A generalization of the submartingale property}        

\author{J\'{a}nos Engl\"{a}nder}


\institute{J\'{a}nos Engl\"{a}nder \at
              Department of Mathematics, University of Colorado, Boulder, CO-80309, USA. \\
              Tel.: +303-492-4846\\
              \email{janos.englander@colorado.edu}   
}

\date{Received: date / Accepted: date}

\maketitle

\begin{abstract}
We generalize the notion of the submartingale property and Doob's inequality. Furthermore, we show how the latter leads to new inequalities for several stochastic processes:  certain time series, 
L\'evy processes, 
random walks, 
processes with independent increments,
 branching processes and continuous state branching processes,  
 branching diffusions and superdiffusions,
 as well as some Markov processes, including geometric Brownian motion. 
\vspace{3mm}

\end{abstract}

\section{Introduction}\label{intro}
The aim of this note is to generalize the notion of the submartingale property and to show how an improvement on Doob's inequality, already pointed out in \cite{ESR}, leads to new inequalities for various stochastic processes, such as 
certain time series, 
L\'evy processes, 
random walks, 
processes with independent increments,
 branching processes and continuous state branching processes, 
 branching diffusions and superdiffusions,
 as well as some Markov processes, including geometric Brownian motion. 
Despite the proofs being very simple, to the best of our knowledge, the inequalities obtained are new. They seem to provide remedies in situations when the standard submartingale toolset is not available.
For background on submartingales and related inequalities see e.g. \cite{RYbook,Stroock.book}.

\medskip
{\bf Notation:} As usual, $\mathbb N$ will denote the nonnegative integers, $\mathbb R_+$ will denote $[0,\infty)$ and SMG will abbreviate `submartingale.' Finally, $E[X;A]$ will denote $E(X\, \mathbf{1}_A),$ and inequalities involving conditional expectations will be meant in the a.s. sense.

\section{Improving Doob's inequality}\label{Improve}
Although the proofs of Theorems A and B below are almost identical to the ones of Theorems 5.2.1 and 7.1.9 in \cite{Stroock.book}, we present their brief proofs for the sake of completeness.

The following maximal inequality is precisely Doob's inequality when $a=1$; when $a<1$, however,  Doob's inequality is not available.
\begin{thma}[Improved Doob; discrete] Let $N\ge 0$. Let $(X_n,\mathcal{F}_n,P)_{0\le n\le N}$ be a discrete stochastic process with the last variable satisfying that  $0\le X_N\in L^1(P)$, and assume that
\begin{equation}\label{mild.assum}
E(X_N\mid \mathcal{F}_n)\ge a X_n
\end{equation}
 holds for all $0\le n< N$ with some $0<a$. Then
$$P\left(\max_{0\le n\le N}X_n\ge \alpha\right)\le\frac{1}{\alpha \widetilde{a}}E\left[X_N;\ \max_{0\le n\le N}X_n\ge \alpha\right],\ \alpha>0,$$ where
$\widetilde{a}:=\min\{a,1\}$.
\end{thma}
\begin{proof}
It is enough to treat the case when $a<1$, otherwise one is simply dealing with Doob's inequality.
Define the mutually disjoint events 
\begin{eqnarray*}
& A_0:=\{X_0\ge \alpha\};\\
& A_n:=\{X_n\ge \alpha\ \mathrm{but}\ \max_{0\le m<n} X_m<\alpha\}\in \mathcal{F}_n,\ n=1,2,...
\end{eqnarray*}
Since $a<1$ and $X_N\ge 0$, the bound \eqref{mild.assum} holds even for $n=N$, and thus
$$P\left(\max_{0\le n\le N}X_n\ge \alpha\right)= 
\sum_{n=0}^{N} P(A_n)\le  \sum_{n=0}^{N}\frac{E[X_n;A_n]}{\alpha}
\le\sum_{n=0}^{N}\frac{E[X_N;A_n]}{\alpha a}
\le\frac{1}{\alpha a}E\left[X_N;\ \max_{0\le n\le N}X_n\ge \alpha\right],$$
as claimed.
\end{proof}
\begin{remark}[$L^p$-inequality]
The standard proof of the $L^{p}$-inequality corresponding to Doob's inequality (see eg. Corollary II.1.6 in \cite{RYbook}) now yields the following, slightly modified result. Let $p>1$ and assume that $E(X_i^p)<\infty$ for $1\le i\le N$. If $X^*_N:=\max_{0\le n\le N}X_n$, then 
$$\|X^*_N\|_p\le \frac{1}{\widetilde{a}}\cdot\frac{p}{p-1}\|X_N\|_p,$$
where $\|\cdot\|_p$ is the $L^p$-norm.

\end{remark}

Next, we treat the continuous counterpart.
\begin{thmb}[Improved Doob; continuous] Let $T>0$.
 Let $(Z_t,\mathcal{F}_t,P)_{t\in [0,T]}$ be a right-continuous stochastic process 
 with the last variable satisfying that  $0\le Z_T\in L^1(P)$,  and assume that 
 \begin{equation}\label{mild.assum.cont}
 E(Z_T\mid \mathcal{F}_t)\ge a Z_t
 \end{equation}
  holds for all $0\le t< T$ where $0<a$. Then, for $\alpha>0$,
$$P\left(\sup_{0\le s\le T}Z_s\ge \alpha\right)\le\frac{1}{\alpha \widetilde{a}}E\left[Z_T;\ \sup_{0\le s\le T}Z_s\ge \alpha\right],$$ where $\widetilde{a}:=\min\{a,1\}$.
\end{thmb}
\begin{proof}
We will write $Z(s)$ instead of $Z_s$ for convenience.
Let $\alpha^*\in (0,\alpha)$.
Let $n\in \mathbb N$ be given and apply Theorem A to the discrete parameter 
process $\left(W_m,\ \mathcal{G}_ {m}, P\right)_{0\le m\le 2^{n}}:=\left(Z\left(\frac{mT}{2^{n}}\right),\ \mathcal{F}_ {\frac{mT}{2^{n}}}, P\right)_{0\le m\le 2^{n}}$ and $N:=2^n$, yielding 
$$P\left(\max_{0\le m \le 2^{n}}W_m\ge \alpha^*\right)\le\frac{1}{\alpha^* \widetilde{a}}E\left[Z_T;\ \max_{0\le m\le 2^{n}}W_m\ge \alpha^*\right]\le \frac{1}{\alpha^* \widetilde{a}}E\left[Z_T;\ \sup_{0\le s\le T}Z_s\ge \alpha^*\right].$$
Exploiting right-continuity, one has 
$$\max_{0\le m \le 2^{n}}W_m=\max_{0\le m \le 2^{n}}Z\left(\frac{mT}{2^{n}}\right)\nearrow\ \sup_{0\le s\le T}Z(s),\ \text{as}\ n\to \infty,\ \text{hence}$$
\begin{eqnarray*}
&P\left(\sup_{0\le s\le T}Z_s>\alpha^*\right)=\
P\left(\lim_n\left\{\max_{0\le m \le 2^{n}}W_m\right\}> \alpha^*\right)
=\lim_n P\left(\max_{0\le m \le 2^{n}}W_m> \alpha^*\right)\\
&\le \frac{1}{\alpha^* \widetilde{a}}E\left[Z_T;\ \sup_{0\le s\le T}Z_s>\alpha^*\right].
\end{eqnarray*}
To complete the proof, let $\alpha^*\uparrow \alpha$, and use monotone convergence for the expectation.
\end{proof}
\begin{remark}[$L^1$ can be dropped] If  the $L^1$-assumption on the last variable fails in the above theorems, then the estimates remain still valid in the sense that the bounds become infinite. In the sequel, we will always use this convention.
\end{remark}

\section{Applications to various processes}\label{Applic}
We now present some useful inequalities which are  applications of Theorems A and B.

\subsection{Application to time series (processes) with step sizes (slopes) bounded from below}
We first consider time series.
 \begin{theorem}[Time series with jump sizes bounded from below]
Let $S=\{S_n\}_{n\ge 0}$ be a sequence of real valued random variables. We may view $S$ as a (not necessarily Markovian) random walk on $\mathbb R$ or as a time series. The only assumption we have about the steps is that $$S_{n+1}-S_n>\ell,\ n\ge 0,\ a.s.,$$ with some $\ell <0$. Then, for $N\ge 0$ and $\alpha\in \mathbb R$, one has
$$P\left(\max_{0\le n\le N}S_n\ge \alpha\right)\le e^{-\alpha+|\ell| N}E\left[e^{S_{N}};\ \max_{0\le n\le N}S_n\ge \alpha\right].$$
\end{theorem}
\begin{proof} Let $X_n:=e^{S_{n}}$ for $n\ge 0$ and apply Theorem A for $\alpha':=e^ \alpha$ and $a:=e^{\ell N}$. 
 \end{proof}
A similar application of Theorem B leads to the following continuous version.
\begin{theorem}[Lower bound on slope] Let  $\ell<0$ and
assume that the right-continuous process $(Y_t,\mathcal{F}_t,P)_{t\ge 0}$  satisfies for all $t>s\ge 0$
that $$\frac{Y_{t}-Y_{s}}{t-s}>\ell,\ P-a.s.$$ Then for $T>0$ and  $\alpha\in \mathbb R$,
$$P\left(\sup_{0\le s\le T}Y_s\ge \alpha\right)\le e^{-\alpha+|\ell|T}E\left[e^{Y_{T}};\ \sup_{0\le s\le T}Y_s\ge \alpha\right].$$
\end{theorem}
\subsection{Application to processes with independent increments} 
If the right-continuous process $(Z_t,\mathcal{F}_t,P)$ on $[0,T]$ has independent increments, then
$$\frac{E(e^{Z_{T}}\mid \mathcal{F}_s)}{e^{Z_{s}}}=E\left(e^{Z_T-Z_s}\mid \mathcal{F}_s\right)=E\left(e^{Z_T-Z_s}\right).$$
Let $$a:=\inf_{0\le s\le T}E(e^{Z_T-Z_s}),$$ and note that clearly $a\le 1$. If $0<a$, then the conditions of Theorem B are satisfied for the process $\widehat Z:=e^Z$.
Therefore, we have 
\begin{theorem}[Independent increments]
If the right-continuous process $(Z_t,\mathcal{F}_t,P)$ on $[0,T]$ has independent increments, and $a:=\inf_{0\le s\le T}E(e^{Z_T-Z_s})>0,$ then for $\alpha\in\mathbb R$,
$$P\left(\sup_{0\le s\le T}Z_s\ge \alpha\right)
\le  \frac {e^{-\alpha}}{a}\, E\left[e^{Z_{T}};\ \sup_{0\le s\le T}Z_s\ge \alpha\right].$$
\end{theorem}
\begin{remark}
If the righthand side is infinite, we still consider the bound valid in the broader sense, and therefore we do not assume any moment condition on $Z_T$.
\end{remark}
As a particular discrete case (of  Theorem A), we let $(S_n,\mathcal{F}_n,P)_{0\le n\le N}$ be a a random walk on $\mathbb Z$ with $S_0=0$. Let the steps $Y_n:=S_{n+1}-S_n$ be independent,  and define $\phi_i:=Ee^{Y_i}$ and $\pi_n:=\Pi_{i=n}^{N-1} \phi_i$. Choosing $$a:=\min_{0\le n\le N}E(e^{S_N-S_n})=\min_{0\le n\le N}\pi_{n},$$ we obtain
\begin{corollary}[Random walks with time-inhomogeneous steps]
For $\alpha\in\mathbb R$,
$$P\left(\max_{0\le n\le N}S_n\ge \alpha\right)
\le  e^{-\alpha}\left(\max_{0\le n\le N}\pi_{n}^{-1}\right)\, E\left[e^{S_N};\ \max_{0\le n\le N}S_n\ge \alpha\right].$$
\end{corollary}
\subsection{Application to L\'evy-processes}
Assuming a little more than just independent increments, namely that $Z$ is  a L\'evy-process, we can get even more appealing estimates.

Recall that a process $Z$ with independent stationary increments is called a L\'evy-process if $Z_0\equiv 0$ and it is continuous in probability, in which case it has a version with almost surely $c\grave{a}dl\grave{a}g$ paths. By the L\'evy-Khintchine Theorem, the distribution of a L\'evy process is characterized by having a specific form for the log-characteristic function, namely
 $$\Psi_t(\theta):=\log \left(E(e^{i\theta Z_{t}})\right)=t\left(ib\theta-\frac{1}{2}\sigma^2 \theta^2-\int_{-\infty}^{\infty}\left(1-e^{i\theta x}+i\theta h(x)\right)\,\Lambda(\mathrm{d}x)\right),\ \theta\in\mathbb R,$$ where  $h(x):=x\mathbf{1}_{|x|\le 1}$,   $\sigma\ge 0$, $b\in\mathbb{R},$ and $\Lambda$ is a measure (called `L\'evy measure') supported on $\mathbb R\setminus \{0\}$ satisfying that $\int_{\mathbb R}\min\{1,x^2\}\, \Lambda(\mathrm{d}x)<\infty$.
The parameters $\sigma,b$ and $\Lambda$ are called the {\it characteristic triple}. 
(For  more background on L\'evy processes, see e.g. \cite{Kypr.book}.)

Since  for $T>0$,
$$a:=\inf_{0\le s\le T}E\left(e^{Z_T-Z_s}\right)=\inf_{0\le s\le T}E\left(e^{Z_{T-s}}\right)=
\inf_{0\le s\le T}\left(E(e^{Z_{1}})\right)^{T-s},$$
  the infimum is either at $0$ or at $T$, and furthermore, denoting $\gamma:=E\left(e^{Z_{1}}\right)\in (0,\infty]$,
\begin{itemize}
\item[$\bullet$] $a=1$ when  $\gamma\ge 1$;
\item[$\bullet$] $a=\gamma^T$  when $\gamma\le 1$.
\end{itemize}
We have obtained that
$a=\min\{1,Ee^{Z_{T}}\}$,
which leads to the the following result. (Again, the righthand sides of the bounds are allowed to be infinite, and so we make no moment assumptions on $Z_T$.)
\begin{theorem}[L\'evy processes]
Consider a L\'evy process $(Z_t,\mathcal{F}_t,P)$ on $[0,T]$ and let $\alpha\in\mathbb R$. Then
$$P\left(\sup_{0\le s\le T}Z_s\ge \alpha\right)\le  e^{-\alpha}\frac{ E\left[e^{Z_{T}};\ \sup_{0\le s\le T}Z_s\ge \alpha\right]}{\min\{1,Ee^{Z_{T}}\}}.$$
In particular, 
$$P\left(\sup_{0\le s\le T}Z_s\ge \alpha\right)\le  e^{-\alpha}\max\{1,Ee^{Z_{T}}\}.$$\end{theorem}
\begin{remark}[Exponential moments]
When $a=1$, that is, $\gamma\ge 1$, the theorem is simply an exponential Doob's inequality. (For example, that is the case for standard Brownian motion.) Nonetheless, when $\gamma< 1$, one obtains a new inequality. 

Assume now that $Ee^{\theta Z_1}<\infty$  for all $\theta\in\mathbb R$. (For example, let $Z_1$ have compound Poisson distribution, such that all exponential moments of the step distribution are finite.) Then, by standard Laplace-transform theory, $M(w):=Ee^{w Z_1}$ is also well defined for all $w\in\mathbb C$, and in this case,  $M(i\theta)=E(e^{i\theta Z_{1}})=e^{\Psi_1(\theta)}$ for $\theta\in \mathbb R$.
Thus, $\gamma =e^{\Psi_1(-i)}<1$ is equivalent to
$$b< -\frac{\sigma^2}{2}+\int_{-\infty}^{\infty}(e^x-1-h(x))\,\Lambda(\mathrm{d}x),$$ 
where $(\sigma,b,\Lambda)$ is the characteristic triple. (Since $\Psi_1(-i)$ is well defined,
$0\le \int_{-\infty}^{\infty}(e^x-1-h(x))\,\Lambda(\mathrm{d}x)<\infty$ must hold.)

\end{remark}

\subsection{Application to subcritical branching processes}
Let $(Z_t)_{t\ge 0}$ be a subcritical branching process, with mean offspring number $0<\mu<1$, and with exponential branching clock with rate $b>0$.

Recall that this means that we start with a single ancestor, and any individual has $X=0,1,2,...$ offspring with corresponding probabilities $p_0,p_1,p_2,...$ (we assume $p_0<1$) and  branching occurs at exponential times with rate $b>0$. Let $h$ be the generating function of the offspring distribution,
$$h(z) :=Ez^{X}= p_0 + p_1 z + p_2 z^2 + ....$$
Then, $$h'(1)=p_1 +2p_2 +3p_3 +...=\mu>0,$$
and subcriticality means that we assume that $\mu<1$.

Suppose  further, that all the offspring of the original single individual also give birth to a random number of offspring, according to the law of $X$,  their offspring do the same as well, and continue this in an inductive manner, assuming that all these mechanisms are independent of each other. 

Let $Z_n$ denote the size of the $n$th generation for $n\ge 0$. (We set $Z_0=1$, as we start with a single particle.) 
The generating function of $Z_n$  satisfies
\begin{equation}\label{iteration}
Ez^{Z_{n}} =h(h(...(z)...)),\ n\ge 1,
\end{equation}
where on the right-hand side one has precisely the $n$th iterate of the function $h$.

 Let $m:=\mu-1\in(-1,0)$.
Since, by the branching property, $E(Z_T\mid Z_s)=e^{bm(T-s)} Z_s$ for $T>s$, we pick $a=e^{bmT} $ and obtain that
\begin{theorem}[Subcritical branching processes]\label{subcr.br}
For $\alpha,T>0$,
$$P\left(\sup_{0\le s\le T}Z_s\ge \alpha\right)\le\alpha^{-1} e^{-bmT}E\left[Z_T;\ \sup_{0\le s\le T}Z_s\ge \alpha\right].$$ 
\end{theorem}
Note: The righthand side is of course bounded by $\alpha^{-1}$ for any $T$ and $\mu<1$, in accordance with the fact that for the $\mu=1$ case, Doob's inequality gives precisely the  $\alpha^{-1}$bound. But if $\alpha$ is large relative to $T$, our bound is much tighter, as the expectation term tends to zero as $\alpha\to\infty$.
\begin{remark}[CSBP's]
For a \emph{continuous state branching process} (CSBP) $X$  with branching mechanism $\beta u-k u^2$ with $\beta<0,\  k>0$, 
we get, by a similar argument, that
$$P\left(\sup_{0\le s\le T}X_s\ge \alpha\right)\le\alpha^{-1} e^{-\beta T}E\left[X_T;\ \sup_{0\le s\le T}X_s\ge \alpha\right],\ \alpha>0.$$ 
A CSBP  can  be thought of as the total mass of a superprocess, see \cite{Kypr.book}
for background on CSBP's. 
\end{remark}
\subsection{Application to time-homogeneous Markov processes}
If $X$ is a time-homogenous Markov process, then  condition \eqref{mild.assum.cont} becomes
$$E_{X_s}(X_T)\ge a X_s,\ s\in[0,T]$$ where $a=a(T)>0$.

Besides the case of the branching process, this inequality is also satisfied, for example, by a geometric Brownian motion $S$ solving the stochastic differential equation
$$\mathrm{d}S_t=\mu S_t \,\mathrm{d}t+\sigma S_t\,\mathrm{d}W_t,$$ 
with $S_0=z>0$.
Here   $\mu\in \mathbb R$, $\sigma>0$, while $W$ is a standard Brownian motion. Indeed,
$$E_{S_{s}} (S_T)=S_se^{\mu (T-s)}\ge aS_s,\ 0\le s\le T,$$
where $a:=1$ for $\mu\ge 0$ and $a:=e^{\mu T}$ for $\mu<0$.

In the latter case for instance, we obtain the following bound.
\begin{theorem}[GBM; $\mu<0$]\label{GBMthm}
Assume that the geometric Brownian motion $S$ has drift $\mu<0$ and $S_0=z$. Then,  for $\alpha>z$,
$$P_z\left(\sup_{t>0}S_t\ge \alpha\right)\le \frac{z}{\alpha}.$$
\end{theorem}
\begin{proof} Let $\alpha^*\in (z,\alpha).$ Using continuity, $$P_z\left(\exists t>0:\ S_t> \alpha^*\right)=\lim_{T\to\infty}P_z\left(\max_{0\le t\le T}S_t>\alpha^*\right).$$ Now, Theorem B along with the previous comments yields for  $T>0$, that
$$P_z\left(\max_{0\le t\le T}S_t\ge  \alpha^*\right)\le \frac{E_z(S_T)}{a\alpha^*}=\frac{ze^{\mu T}}{e^{\mu T}\alpha^*}=\frac{z}{\alpha^*},$$  
hence
$P_z\left(\exists t>0:\ S_t> \alpha^*\right)\le\frac{z}{\alpha^*},$
and we are done by letting $\alpha^*\uparrow \alpha$.
 \end{proof}
For some related results on geometric Brownian motion, see \cite{GP}.

\section{Application to proving limits}
\subsection{Almost sure convergence}
The following situation is typical for limit theorems. Suppose that one is working with  a process that is defined for continuous times, and wishes to prove a limit theorem for large times. Often, one then must  go through a rather unpleasant two-step procedure consisting of 
\begin{itemize}
\item working with a `discrete time skeleton' first, 
\item upgrading the result to all times next.
\end{itemize}
 (A classic paper addressing these kind of issues is \cite{K1963}.) The following result offers a method to solve this problem. 
 \begin{theorem}[Almost sure convergence]\label{a.s.conv}
Let $(X_t,\mathcal{F}_s,P)_{t\ge 0}$ be a nonnegative real valued, filtered stochastic process, such that $EX_t<\infty$ for all $t\ge 0$. Assume that for a sufficiently small $T>0$ the following holds:
\begin{enumerate}
\item there is an $a\in (0,1]$ such that
$$E(X_{nT+t}\mid \mathcal{F}^{(n)}_s)\ge a X_{nT+s},\ \forall n\ge 1,\ \forall\ 0\le s<t<T,$$
where $\mathcal{F}^{(n)}_s:=\sigma(X_{nT+r}:\ r\in[0,s))$;
\item $\sum_{n} EX_{nT}<\infty$.
\end{enumerate}
Then $\lim_{t\to\infty} X_t=0$ holds $P$-a.s.
\end{theorem}
\begin{proof}
By the Borel-Cantelli Lemma, it is enough to show that for any given $\epsilon >0$,
$$\sum_{n\ge 0} P\left(\sup_{s\in [0,T]} X_{nT+s}>\epsilon\right)<\infty.$$
By our first assumption along with Theorem B (applied to $Y^{(n)}_t:=X_{nT+t}$ on $[0,T]$), the lefthand side is bounded by $(a\epsilon)^{-1} \sum_{i\ge i_{0}}EX_{(n+1)T}$, and we are done, given  our second assumption.
 \end{proof}
\subsection{Examples of applications of Theorem \ref{a.s.conv}} We present a few applications of Theorem \ref{a.s.conv} below.We start with some notation.
\begin{notation}  In the sequel, $\mathcal{M}(\mathbb R^d)$ denotes the space of finite measures on $\R^d$;
for $i\ge 1$ and $\eta\in (0, 1]$,
 $C^{i,\eta}(\R^d)$ denotes the space of $i$ times
continuously differentiable functions with all their $i$-th order
derivatives belonging to $C^{\eta}(\R^d)$. (Here $C^{\eta}(\R^d)$
denotes the usual H\"older space.) 
\end{notation}

\subsubsection{Subcritical branching and GBM} Consider the subcritical branching process $Z$ in Theorem \ref{subcr.br} and the geometric Brownian motion in Theorem \ref{GBMthm} with $\mu<0$.  For these processes, the  summability of the expectations at integer times is obvious. Hence,  both  tend to zero as $t\to\infty$, almost surely.
 
\subsubsection{Total mass of superprocesses} A more involved case is  the proof of the fact that the `over-scaled' total mass of a superprocess tends to zero. Below we give some background on the model for the non-expert reader (for the result and its proof and for more background, see \cite{ESR}).

Consider $Y=\{Y_t;\,t\geq 0\}$, the diffusion process with
probabilities $\{{P}{_{x}},\ x\in \R^d\}$ and expectations $\{{E}{_{x}},\ x\in \R^d\}$ corresponding to $L$ on $\R^d$, where
$$
L:=\frac{1}{2}\nabla\cdot
a\nabla+b\cdot\nabla\quad \mbox{ on }\R^d,
$$
and $a, b$ satisfy the following
\begin{description}
\item{(1)} the symmetric matrix $a=\{a_{i,j}\}$  satisfies
$$
A_1|v|^2\le\sum^d_{i,j=1}a_{i,j}(x)v_iv_j\le A_2|v|^2,\quad\mbox{ for
all }v\in\R^d\mbox{ and }x\in \R^d
$$
with some  $A_1, A_2>0$, and
 $a_{i,j}\in C^{1,\eta}, i,j=1,\cdots,d,$ for
some $\eta$ in $(0,1]$;
\item{(2)} the coefficients $b_i$, $i=1, \cdots, d$, are measurable functions satisfying
$$
\sum^d_{i=1}|b_i(x)|\le C(1+|x|), \qquad \mbox{ for
all } x\in \R^d
$$
with some $C>0$;
\item{(3)} there
exists a differentiable function $Q:\R^d\to \R$ such that $b=a\nabla
Q$.
\end{description}
An {\it $(L,\beta, k)$-superprocess}  is an $\mathcal{M}(\mathbb R^d)$-valued Markov process
$(\{X_t\}_{t\ge 0};\P_{\mu},\,\mu\in \mathcal{M}(\mathbb R^d))$ such that
$\P_{\mu}(X_0=\mu)=1$, and satisfying that  for any  bounded  Borel  $f\ge 0$ on $\R^d$,  
\begin{equation}\label{inhom-fund}
\P_{ \mu}\exp\langle -f,
X_{t}\rangle=\exp\langle -u(t, \cdot),\mu\rangle,
\end{equation} where scalar products denote integration, and with some sufficiently nice (see \cite{ESR}) functions $k\ge 0$ and $\beta$, the function $u$ is the minimal nonnegative solution to
\begin{equation}\label{inhom-int}
u(t, x)+E_{x}\int^{t}_0k(\xi_s)(u(t-s, \xi_s))^2\mathrm{d}s-E_{
x}\int^{t}_0\beta(\xi_s)u(t-s, \xi_s)\mathrm{d}s=E_{x}f(\xi_{t}).
\end{equation}
In particular, $\|X_t\|=\langle 1,X_t\rangle$ is the total mass of the superprocess. In \cite{ESR} it has been proven that if $\lambda\in\R$ is sufficiently large, then $\P_{\mu}(\lim_{t\to\infty}e^{-\lambda t}\|X_t\|=0)=1$ for $\mu\in \mathcal{M}(\mathbb R^d)$.
Although Doob's inequality is not applicable in this situation,  Theorem \ref{a.s.conv} works. (For the details the reader should consult \cite{ESR}, but in fact the proof is similar to that of Theorem \ref{over.sc} below.)

\subsubsection{Total population in a branching diffusion} Let $D\subseteq \R^d$ be a non-empty domain and $$L:=\frac{1}{2}\nabla \cdot  a\nabla+ b\cdot\nabla\ \text{on}\ D,$$
where the functions $  a_{i,j},  b_i:D\to \R,\ i,j=1,...,d$, belong to $C^{1,\eta}(D),\ \eta\in(0,1]$, and the symmetric matrix $( a_{i,j}(x))_{1\le i,j\le d}$ is positive definite for all $x\in D$.  Consider $Y=\{Y_t;\,t\geq 0\}$, the diffusion process with
probabilities $\{{P}{_{x}},\ x\in D\}$ and expectations $\{{E}{_{x}},\ x\in D\}$ corresponding to $L$ on $D$. We do not assume that $Y$ is conservative, that is, for $\tau_D:=\inf\{t\ge 0\mid Y_t\not\in D\}$, the exit time from $D$, $\tau_D<\infty$ may hold with positive probability. Intuitively, this means that $Y$ may get killed at the Euclidean boundary of $D$ or `run out to infinity' in finite time.

Let us first assume that 
\begin{equation}\label{nagyon.szep.beta}
0\leq \beta \in C^{\eta }(D),\ \sup_D \beta<\infty,\ \beta \not\equiv 0.
\end{equation} 
The (strictly dyadic) \emph{$(L,\beta ;D)$-branching diffusion} is the
Markov process $Z$ with motion component $Y$ and with spatially dependent rate  $\beta $, replacing particles by precisely two offspring when branching and starting from a single individual. Informally, starting  with an initial particle at $x\in D$, it performs a diffusion corresponding to $L$ (with killing at $\partial D$) and the probability that it does not branch until $t>0$ given its path $\{Y_s;0\le s\le t\}$ is $\exp(-\int_0^t \beta (Y_s)\,\mathrm{d}s)$. When it does branch, it dies and produces two offspring, each of which follow the same rule, independently of each other and of the parent particle's past, etc. (Already at the instant of the branching we  have two offspring particles at the same location, i.e. at the location of the death of their parent.) Write $\mathbb P_x$ (instead of the more correct $\mathbb P_{\delta_{x}}$) for  the probability when $Z$ starts with a single particle at $x\in D$

Then $Z$ can be considered living either on the space of `point configurations,' that is, sets which consist of finitely many (not necessarily different) points in $D$; or $\mathcal{M}(D),$ the space of finite discrete measures on $D$.
We will write  $\langle f,Z_t \rangle:=\sum_1^{N_{t}}f({Z_{t}^i})$, where $N_t=\|Z_t\|$ is the number of points (with multiplicity) in $D$  at time $t$.

Now relax the assumption that  $\sup_D \beta<\infty$ and replace it with the less stringent one that $\beta$ is in the Kato-class $\mathbf{K}(Y)$, meaning that $$\lim_{t\downarrow0}\sup_{x\in D}P_x\left(\int^t_0|\beta(Y_s)|\,\mathrm{d}s\right)=0.$$
Define
$$
\lambda_\infty(\beta) :=\lim_{t\to\infty}\frac{1}{t}\log\|S^\beta_t\|_{\infty},
$$
where $S^\beta$ is the semigroup corresponding to the operator $L+\beta$ on $D$.
We call $\lambda_{\infty}=\lambda_{\infty}(\beta)$ the {\it $L^{\infty}$-growth bound}.
The Kato-class assumption implies, in fact, that the semigroup is well defined and that $\lambda_\infty(\beta)<\infty$. (See \cite{ESR}). 

By standard theory then (see 1.14 in \cite{Ebook}), the {\it generalized principal eigenvalue} of $L+\beta$ on $D$,  $\lambda_c(\beta):=\inf\{\lambda\in\mathbb R\mid \exists u>0\ \text{s.t}\ (L+\beta-\lambda)u=0\ \text{in}\ D \}$ satisfies that $\lambda_c(\beta)\le \lambda_\infty(\beta)$ and thus $\lambda_c(\beta)<\infty$. 

Whenever  $\lambda_c(\beta)<\infty$, $Z$ is well defined as a locally finite (discrete) measured-valued process even if $\beta$ is not bounded from above \cite{Ebook}.
But since we even assume that $\lambda_\infty(\beta)<\infty$, we know that the process is almost surely finite measure valued, not just locally, but globally. This is because of the well known fact (called `Many-to-one formula'; see e.g. \cite{Ebook}) that
$\mathbb E_{\mu} \|Z_t\|=\langle P_t^\beta 1,\mu\rangle$ for $t\ge 0$, which implies that even the expectation of the total mass is finite. For the growth of the total mass, we now derive a bound using Theorem \ref{a.s.conv}.

\begin{theorem}[Over-scaling]\label{over.sc} Let $\mu$ be a nonempty finite discrete point measure. If $\lambda>\lambda_\infty$ then
$$\lim_{t\to\infty}e^{-\lambda t}\|Z_t\| =0,\ \mathbb P_\mu\text{-a.s.}$$

\end{theorem}
\begin{proof}
For $n\ge 0$ and $s>0$, let   $$\mathcal{F}^{(n)}_s:=\sigma(X_{nT+r}: r\in [0, s]).$$ By Theorem \ref{a.s.conv} (applied to the process $X$ with $X_t:=e^{-\lambda t}\|Z_t\|$)  it is enough to verify these two statements:
\begin{enumerate}
\item For some $a\in(0,1)$ and $T>0$, $$\mathbb E_\mu (e^{-\lambda(nT+t)}\|Z_{nT+t}\|\mid \mathcal{F}^{(n)}_s)\ge a e^{-\lambda(nT+s)}\|Z_{nT+s}\|$$ holds for $n\ge 1$ and $0<t\le T$;
\item $\sum_n \mathbb E_\mu e^{-\lambda nT}\|Z_{nT}\|<\infty.$ 
\end{enumerate}
The second statement simply follows from the facts that $\lambda>\lambda_\infty$, while $$\mathbb E_\mu\|Z_{nT+t}\|=\exp(\lambda_{\infty}nT+o(n))$$ as $n\to\infty$.

The first statement is a consequence of the Kato-class assumption. Indeed, using the Markov and the branching properties,
\begin{equation}\label{Mbr}
\mathbb E_\mu (e^{-\lambda(nT+t)}\|Z_{nT+t}\|\mid \mathcal{F}^{(n)}_s)=
 e^{-\lambda(nT+t)}\mathbb E_{Z_{nT+s}}\|Z_{t-s}\|= \langle e^{-\lambda(nT+t)}\mathbb E_{x}\|Z_{t-s}\|,Z_{nT+s}\rangle .
 \end{equation}
Fix an arbitrary $a\in (0,1)$; we are now going to determine $T$ that works for this given $a$.
Since  $\beta\in\mathbf{K}(Y)$, i.e.  $$\lim_{t\downarrow 0}
\sup_{x\in \R^d}P_x\int_0^t |\beta| (Y_s)\, \mathrm{d}s=0,$$ we are able to pick a $T>0$ such that
$$
-\lambda t +P_x \int_0^t \beta (Y_s)\, \mathrm{d}s\ge \log a,
$$
for all $0<t< T$ and all $x\in \R^d$.
By Jensen's inequality,
$$
-\lambda t+\log P_x \exp\left(\int_0^t \beta (Y_s)\, \mathrm{d}s\right)\ge \log a,
$$
and thus
$$
\mathbb E_{\delta_{x}}e^{-\lambda t}\|X_t\|=e^{-\lambda t}P_x \exp\left(\int_0^t
\beta (Y_s)\, \mathrm{d}s\right)\ge a
$$
holds too, for all $0<t< T$ and all $x\in \R^d$.
Therefore we can continue \eqref{Mbr} with
$$\ge a e^{\lambda(t-s)}e^{-\lambda(nT+t))}\|Z_{nT+s}\|=a e^{-\lambda(nT+s)}\|Z_{nT+s}\|,$$
and we are done.
\end{proof}



\section{SMG's and $a$-achieving processes}
So far we have explored some consequences of Theorems A and B.  In Theorem A we only compared $X_N$  to all $X_n,n<N$ with $N$ fixed. If we compare {\it all} the pairs of the random variables, then we can define a new  class of stochastic processes which we dub `$a$-achieving processes.'
\begin{definition}[$a$-achieving process]
Let $a>0$.
\begin{itemize}
\item[(a)] Let $M\in\mathbb N\cup \{+\infty\}$. 
We call an integrable stochastic process $X=\{X_n\}_{n\in \mathbb N,  n\le M}$ \emph{$a$-achieving} if 
\begin{equation}\label{achieve.disc}
E(X_{n+1}\mid \mathcal{F}_n)\ge a X_n
\end{equation}
holds for every   $n\in\mathbb N$ satisfying $n\le M$; we call it \emph{uniformly $a$-achieving} if
\begin{equation}\label{unif.achieve.disc}
E(X_{n}\mid \mathcal{F}_m)\ge a X_m
\end{equation}
holds for all $m,n\in\mathbb N$ such that $0\le m<n\le M$.

\item[(b)] Let $S\in\mathbb R_{+}\cup \{+\infty\}$. 
We call a right-continuous integrable stochastic process $X=\{X_t\}_{t\in\mathbb R_{+},t\le S}$ \emph{$a$-achieving} if 
\begin{equation}\label{achieve.cont}
E(X_t\mid \mathcal{F}_s)\ge a^{t-s} X_s
\end{equation}  holds for every pair  $s,t\in \mathbb R_{+}$ satisfying $s<t\le S$;  we call it \emph{uniformly $a$-achieving} if
\begin{equation}\label{unif.achieve.cont}
E(X_{t}\mid \mathcal{F}_s)\ge a X_s
\end{equation}
holds for every pair  $s,t\in \mathbb R_{+}$ satisfying $s<t\le S$.
\end{itemize}
\end{definition}
\begin{example} The subcritical branching process $Z$ in Theorem \ref{subcr.br} is $e^{bm}$-achieving, while the geometric Brownian motion in Theorem \ref{GBMthm} is $e^{\mu}$-achieving.
\end{example}
An equivalent definition is as follows.
 \begin{lemma}\label{backtrafo}
 Let $a>0$.
 \begin{itemize}
 \item[(a)]	 (Discrete) $X$ is $a$-achieving if and only if $Y$ defined by $Y_n:=a^{-n}X_n$ is a submartingale.  
 
\item[ (b)] (Continuous) $X$ is $a$-achieving if and only if $Y$ defined by $Y_t:=a^{-t}X_t$ is a (right-continuous) submartingale.
\end{itemize}
 \end{lemma}
 \begin{proof} (a) First assume that $X$ is $a$-achieving. It is easy to prove then by induction that
 $E(X_n\mid \mathcal{F}_m)\ge a^{n-m}X_m,\ n\ge m.$
 Hence, $E(Y_n\mid \mathcal{F}_m)\ge a^{-n}a^{n-m}X_m=a^{-m}X_m=Y_m.$

Conversely, if $Y$ is a submartingale then
$a^{-m-1}E(X_{m+1}\mid \mathcal{F}_m)=E(Y_{m+1}\mid \mathcal{F}_m)\ge Y_m=a^{-m}X_m,$
so $E(X_{m+1}\mid \mathcal{F}_m)\ge aX_m.$
 
\medskip
\noindent (b)  Let first $X$ be $a$-achieving. Using the  definition,   $E(Y_t\mid \mathcal{F}_s)\ge a^{-t}a^{t-s}X_s=a^{-s}X_s=Y_s,$ for all  $s,t\in \mathbb R_{+},\ s<t\le S.$
Conversely, if $Y$ is a submartingale then
$a^{-t}E(X_{t}\mid \mathcal{F}_s)=E(Y_{t}\mid \mathcal{F}_s)\ge Y_s=a^{-s}X_s,$
yielding  $E(X_{t}\mid \mathcal{F}_s)\ge a^{t-s}X_s.$
 \end{proof}

A convenient property of submartingales is that their class is closed under transformations with non-decreasing and convex functions. We are now generalizing this property.
 In order to accomplish this, we are going to work with functions which are {\it approximately convex}.
Concerning this notion, we  briefly explain the basic facts below; the interested reader may check e.g. \cite{HU,P} and the references therein for more elaboration.
\begin{definition}[Approximate convexity] Let $I\subseteq \mathbb R$ be a (bounded or unbounded) interval and $\delta\ge 0$. The function 
$f: I\to\mathbb R$ is called \emph {$\delta$-convex} if
$$f(tx+(1-t)y)\le tf(x)+(1-t)f(y)+\delta$$
holds for $x,y\in I$ and $t\in [0,1]$. (True convexity in particular means that $\delta=0$ can be taken.)
\end{definition}
  \begin{theorem}[From SMG to uniformly $a$-achieving]\label{trafo}
Let $I$ be a (bounded or unbounded) interval and $X$ an $I$-valued submartingale. Assume that $f:I\to\mathbb R$ is a non-decreasing $\delta$-convex function with $\delta\ge 0$, and in the continuous setting assume also that $f$ is continuous. Then the process $Y:=e^{f(X)}$ is uniformly $e^{-\delta}$-achieving.
\end{theorem}
In order to prove Theorem \ref{trafo} we need two lemmas.
\begin{lemma}[Hyers-Ulam]
$f: I\to\mathbb{R}$ is $\delta$-convex if and only if
it decomposes as $f=g+h$, where $g$ is a convex function on $I$ and $\sup_{x\in I} |h(x)|\le \delta/2.$
\end{lemma}
\begin{proof} This is a particular case of the Hyers-Ulam Theorem \cite{HU,P}.
 \end{proof}
As a corollary we get the next result.
\begin{lemma}[Approximate Jensen]\label{apprJ}
If $f: I\to\mathbb{R}$ is $\delta$-convex and $X$ is an $I$-valued random variable in $L^1$, then
$$Ef(X)\ge f(E(X))-\delta.$$
\end{lemma} 
\begin{proof}
Consider the  Hyers-Ulam decomposition, $f=g+h$. We have by Jensen's inequality that
$$Ef(X)=Eg(X)+Eh(X)\ge g(E(X))-\delta/2=f(E(X))-h(E(X))-\delta/2\ge f(E(X))-2 \delta/2,$$
as claimed.
\end{proof}
\begin{proof} (of Theorem \ref{trafo})
We treat the discrete case; the continuous case is very similar.
 
Let us `exponentiate'  Lemma \ref{apprJ}. That is, for $F:=e^f$, 
$$E[F(X)\mid \mathcal{F}_{m}]=E[e^{f(X)}\mid \mathcal{F}_{m}]\ge e^{E[f(X)\mid \mathcal{F}_{m}]}\ge  e^{f(E[X\mid \mathcal{F}_{m}])}e^{-\delta}=F(E[X\mid \mathcal{F}_{m}])e^{-\delta},$$
where the first inequality uses the conditional Jensen's inequality for $Y:=f(X)$, and the second inequality exploits Lemma \ref{apprJ} for the conditional expectation. (This is fine because in the proof of Lemma \ref{apprJ}, one can use conditional Jensen too for $g$.)
Now, to see that $Y$ defined by $Y_n=\exp[f(X_{n})]=F(X_n)$ is uniformly $e^{-\delta}$-achieving, replace $X$ by $X_n$, where $n\ge m\ge 0$. Then one has
$$E(F(X_{n})\mid \mathcal{F}_{m})\ge e^{-\delta} F(E((X_{n})\mid \mathcal{F}_{m}))\ge e^{-\delta} F(X_m),\ n\ge m\ge 0,$$
where the last step relies on the submartingale assumption and monotonicity. 
 \end{proof}
 \begin{remark} Note that in the Hyers-Ulam decomposition, the convex function $g$ is not necessarily non-decreasing, hence $g(X)$ and $e^{g(X)}$ are not necessarily  submartingales, preventing one from using Doob's inequality. 
 \end{remark}
Taking the  composition of the two transformations appearing in Theorem \ref{trafo} and Lemma \ref{backtrafo} (from $a$-achieving to SMG to $a$-achieving, or from SMG to $a$-achieving to SMG), we immediately get the following invariance results, stated, for simplicity, in the discrete case. (In the continuous case $f$ must be continuous and the processes must be right-continuous, as well.)
 \begin{theorem}[Invariance]
 Let $\delta\ge 0$ and $f$ be a non-decreasing $\delta$-convex function on $I$.
 
 (i) If $X$ is an $I$-valued submartingale then so is $Y$, where $Y_n:=\exp(\delta n+f(X_n))$, provided it is integrable.
 
 (ii)  If $X$ is $a$-achieving, then $U$ defined by $U_n:=e^{f(X_n/a^n)}$ is uniformly $e^{-\delta}$-achieving, provided $X_n(\omega)/a^n\in I$ for all $\omega\in\Omega,n\ge 0$.
 \end{theorem}
 \begin{proof}
 The claims follow from Theorem \ref{trafo} and Lemma \ref{backtrafo}.
 \end{proof}
 {\bf Acknowledgement.} The author is grateful to an anonymous referee for his/her close reading of the manuscript and for pointing out some glitches.

\end{document}